\documentclass[reqno]{amsart}
\usepackage{amssymb}
\usepackage{amsthm}
\usepackage{bbm}
\usepackage{graphicx}

\usepackage[colorlinks=true]{hyperref}
\usepackage[all]{xypic}

\usepackage[final]{epsfig}
\usepackage{psfrag}
\usepackage{epstopdf}

\newtheorem{thm}{Theorem}[section]
\newtheorem{theorem}[thm]{Theorem}

\newtheorem{corollary}[thm]{Corollary}
\newtheorem{lemma}[thm]{Lemma}

\newtheorem{prop}[thm]{Proposition}

\theoremstyle{definition}
\newtheorem{definition}{Definition}[section]

\newtheorem{example}{Example}[section]

\newcommand{\M}{\mathcal{M}}

\newcommand{\defin}[1]{{\it #1}}

\newcommand{\N}{\mathbb{N}}

{\bf}{\it}
{\bf}{\it}
\newtheorem*{riemann-mapping-theorem}{Riemann Mapping  Theorem}{\bf}{\it}
{\bf}{\it}
{\bf}{\it}
{\bf}{\it}
{\bf}{\it}
{\bf}{\it}
{\bf}{\it}
{\bf}{\it}
{\bf}{\it}
\newtheorem*{carker}{Carath\'eodory Kernel Theorem}{\bf}{\it}
{\bf}{\it}
\newtheorem*{bpt}{Beurling Projection Theorem}{\bf}{\it}

\newenvironment{pf*}[1]{\proof[#1]}{\endproof}
\usepackage{euscript}

\newcommand{\beq}{\begin{equation}}
\newcommand{\eeq}{\end{equation}}

\newtheorem{defn}{Definition}[section]

\newcommand{\riem}{{\hat{\CC}}}

\newcommand{\dist}{\operatorname{dist}}

\newcommand{\eps}{\epsilon}

\numberwithin{equation}{section}

\newcommand{\supp}{\operatorname{Supp}}

\newcommand{\PP}{{\mathbb P}}
\newcommand{\CC}{{\mathbb C}}
\newcommand{\RR}{{\mathbb R}}
\newcommand{\EE}{{\mathbb E}}

\newcommand{\NN}{{\mathbb N}}
\newcommand{\DD}{{\mathbb D}}

\newcommand{\hC}{{\hat{\mathbb C}}}

\newcommand{\ignore}[1]{{}}

\title{Carath\'eodory convergence and harmonic measure}
\author{Ilia Binder}

\address{Department of Mathematics, University of Toronto, Bahen Centre, 40 St. George St., Toronto, Ontario, CANADA M5S 2E4}

\email{ilia@math.toronto.edu}

\thanks{I.\ B.\ was supported in part by an NSERC Discovery grant.}
\author{Cristobal Rojas}

\address{Departmento de Matematicas,
Universidad Andres Bello.
Republica 498, 2do piso,
Santiago, Chile.}

\email{crojas@mat-unab.cl}
\author{Michael Yampolsky}

\address{Department of Mathematics, University of Toronto, Bahen Centre, 40 St. George St., Toronto, Ontario, CANADA M5S 2E4}

\email{yampol@math.toronto.edu}

\thanks{M.\ Y.\ was supported in part by an NSERC Discovery grant.}
\begin{document}
\begin{abstract}
We give several new characterizations of Carath\'eodory convergence of simply connected domains. We then investigate how different definitions of convergence generalize to the multiply-connected case.
\end{abstract}
\date{\today}
\maketitle

\section{Introduction}

The motivation for this paper came from the continued series of works of the authors on developing constructive Complex Analysis (cf. \cite{BBY-riemann,BRY-car}). Computable analysis is based on the notion of {\it approximability}. For instance, a domain can be approximated by a nested sequence of interior polygonal approximations, which can be used to approximate its Riemann mapping by piecewise-linear maps (see e.g. \cite{hertling}) -- both in theory, and in computational practice. The harmonic measure of a domain can be approximated by a weakly converging sequence of finitely supported measures, which can be computed given an approximation of the domain (cf. \cite{BBRY}); and so on.

This point of view leads naturally to consider the relationships between various notions of convergence used in Complex Analysis, which was our starting point. It has led us to realize that various standard notions of convergence used for simply-connected domains (i.e. the convergence of Riemann maps, Green's functions, harmonic measures, etc) are all equivalent. This is seemingly a new observation, and its formulation and proof constitute the first part of this note.

The next natural step was to see how these notions of convergence disagree in the case of general non-simply connected domains; and what conditions can be imposed on such domains to reconcile them. This discussion forms the second half of this note.

We hope that our observations will be of an independent interest. In fact, it is surprising to us that they have apparently not been made before, since they concern some of the basic notions of geometric Complex Analysis.

\subsection{Basic properties of harmonic measure}
A detailed discussion of harmonic measure can be found in \cite{Garnett-Marshall}. Here we briefly recall some of the relevant facts.

A domain $\Omega$ in $\hat\CC$ is called {\it hyperbolic} if its complement $K=\hat\CC\setminus\Omega$ contains at least two points.
Let $\Omega$ be a simply-connected hyperbolic domain.
 The {\it harmonic measure} of $\Omega$ at $w$, denoted
$\omega_\Omega(w,\cdot)$, is defined on the boundary $\partial \Omega$. For a set $E\subset \partial \Omega$ its harmonic measure $\omega_\Omega(w,E)$ is equal to the probability that
a Brownian path originating at $w$ will first hit $\partial \Omega$ within the set $E$.
The support
$$\text{supp}\omega_\Omega(w,\cdot)=\partial\Omega.$$

If, for example, the boundary is locally connected, then the conformal Riemann mapping
$$f:\Omega\to\DD,\;f(w)=0,\;f'(w)>0$$
continuously extends to the unit circle by Carath\'eodory Theorem. By conformal invariance of Brownian motion, the
harmonic measure $\omega_\Omega(w,\cdot)$ is the pull-back $f^*\lambda$ of the harmonic measure of the unit disk at $0$; by symmetry, $\lambda$ is the Lebesgue measure on the unit circle.

To define the harmonic measure for a connected, but non simply-connected hyperbolic domain $\Omega=\hat\CC\setminus K$ we have to require that  logarithmic capacity $\text{Cap}(K)$ is positive: this ensures that a Brownian path originating in $\Omega$
will hit $\partial \Omega$ almost surely. The harmonic measure of a set $E\subset\partial\Omega$ is then again the probability of a Brownian path originating at $w$ to hit the boundary $\partial\Omega$ inside $E$.
In this case, the support of the harmonic measure is
$$\supp\omega_\Omega(w,\cdot)=\text{Reg}(\Omega),$$
where $\text{Reg}(\Omega)$ is the closure of the set of regular points of the boundary of $\Omega$.
Let $\Omega^*$ be the connected component of $\hat\CC\setminus \text{Reg}(\Omega)$ which contains $\Omega$: we will call this domain the {\it regularization} of $\Omega$. For any choice of $w\in\Omega$, the domains $\Omega$ and $\Omega^*$ possess the same harmonic measures. We will say that a domain $\Omega$ is {\it regular} if
$$\Omega=\Omega^*.$$

The proof of the following classical result can be found in \cite{Garnett-Marshall}:
\begin{bpt}\label{thm:bpt}
Let $K$ be a closed subset of $\bar\DD\setminus\{ 0\}$, and $K^*\equiv \{|w|\;:\; w\in K\}$ is the circular projection of $K$, then for every $z\in\DD\setminus K$
$$\omega_{\DD\setminus K}(z,K)\geq \omega_{\DD\setminus K^*}(-|z|,K^*).$$
\end{bpt}

\begin{defn}
\label{defn:unifperf}
We recall that a compact set $K\subset\riem$ which contains at least two points is {\it uniformly perfect} if the moduli of the ring domains
separating $K$ are bounded from above. Equivalently, there exists some $C>0$ such that for any
$x\in K$ and $r>0$, we have
$$\left(D(x,Cr)\setminus D(x,r)\right)\cap K=\emptyset\implies K\subset D(x,r).$$ In particular, every connected set is uniformly perfect.

\end{defn}

\noindent
Uniform perfectness for the planar sets $K$ implies the following (see  Theorem 1, \cite{pom79}):

\begin{prop}
\label{first hit}
There exists a constant $\nu=\nu(C)$ (with $C$ as in Definition~\ref{defn:unifperf})
such that for any $\eta>0$ the following holds. Let $y\in \Omega$ be a point such that $\dist(y,\partial\Omega)\le \eta/2$, and
let $B^y$ be a Brownian Motion started at $y$. Let $$T^y:=\min\{t~:~B^y_t\in \partial \Omega\}$$ be the first time
$B^y$ hits the boundary of $\Omega$. Then
\beq
\label{eq:cap}
\PP[|B^y_{T_y}-y|\geq \eta]<\nu.\eeq
\end{prop}
 In other words, there is at least a constant probability that the first point where $B^y$ hits the boundary
 is close to the starting point $y$.

\subsection{Weak convergence of measures}Let $\M(X)$ denote the set of Borel probability measures over a metric space $X$, which we will assume to be compact and separable. We recall the notion of
weak convergence of measures:

\begin{definition}
A sequence of measures $\mu _{n}\in \M(X)$ is said to be \defin{weakly convergent} to $\mu\in \M(X) $ if $\int f d\mu_{n}\rightarrow \int f d\mu$ for each $f\in C_{0}(X)$.
\end{definition}

Any smaller family of functions characterizing the weak convergence is called \defin{sufficient}.  It is well-known, that when  $X$ is a compact separable and complete metric space, then so is $\M(X)$. In this case, weak convergence on $\M(X)$ is compatible with the notion of \defin{Wasserstein-Kantorovich distance}, defined by:

\begin{equation*}
W_{1}(\mu,\nu)=\underset{f\in 1\text{-Lip}(X)}{\sup}\left|\int f d\mu-\int f d\nu\right|
\end{equation*}
\noindent where $1\mbox{-Lip}(X)$ is the space of Lipschitz functions on $X$, having Lipschitz constant less than one.

\section{Caratheodory convergence for simply connected domains}

\subsection{Classical definitions}
Let  $\Omega$ be  a simply connected domain with a marked point  $w \in \Omega$, and let $\Omega_{n}$ be a sequence of simply connected domains with marked points $w_n\in\Omega_n$. The classical notion of Caratheodory convergence is defined as follows.

\begin{definition}\label{s.c. convergence} Let $w\in \Omega$. We say that $(\Omega_{n},w_n)$ converges to $(\Omega,w)$  in  the \defin{Carath{\'e}odory sense} if the following holds:
\begin{itemize}
\item $w_n\to w$;
\item for any compact $K \subset \Omega$ and all $n$ large enough, $K \subset \Omega_{n}$;
\item for any open connected set $U\ni w$, if $U\subset \Omega_{n}$ for infinitely many $n$, then $U\subset \Omega$.
\end{itemize}
\end{definition}
The classical result of Carath{\'e}odory \cite{Car12,Pom75} states:
\begin{carker} Let $(\Omega_n,w_n)$, and $(\Omega,w)$ be as above. Consider the conformal Riemann parametrizations $\phi_n: (\DD,0)\mapsto(\Omega_n, w_n)$ with $\phi_n'(0)>0$ and $\phi: (\DD,0)\mapsto(\Omega, w)$ with $\phi'(0)>0$. Then
 the Carath{\'e}odory convergence of $(\Omega_n,w_n)$ to $(\Omega,w)$  is equivalent to the uniform convergence of $\phi_n$ to $\phi$ on compact subsets of $\DD$.
\end{carker}

\noindent
It is worth noting, that a similar statement holds for the Riemann mappings $f_n\equiv \phi_n^{-1}$:
\begin{prop}
\label{prop:fconv}
Let $f_n=\phi_n^{-1}$ and $f\equiv \phi^{-1}$. The domains $(\Omega_n,w_n)\to(\Omega,w)$ in the Carath\'eodory sense  if and only if for every compact subset $K\Subset \Omega$ the following hold:
\begin{itemize}
\item for all $n$ large enough, $K\Subset \Omega_n$;
\item $f_n\rightrightarrows f$ on $K$.
\end{itemize}
\end{prop}
\begin{proof}
To prove the "if" direction, let $D_m=D(0,1-1/m)$, and let $K_m=\phi(D_m)$. Let $\gamma$ be a boundary of a Jordan subdomain $\Omega'$ with $K_{m-1}\Subset\Omega'\Subset K _m$. Then the winding number of $f(\gamma)$ around every $z\in D_{m-1}$ is $1$. By the uniform convergence on $\gamma$, the same is true for $n\geq n_0$ large enough. Thus for these $n$, $f_n(K_m)\supset f_n(\Omega')\supset D_{m-1}$. By Cauchy Theorem, for such $n$, the maps $\phi_n$ are equicontinuous on $D_{m-1}$. By  Arzel{\`a}-Ascoli theorem, it follows that $\phi_n\rightrightarrows \phi$ on $D_{m-1}$. Since every compact in $\DD$ is a subset of $D_m$ for some $m$, the claim follows.

The other direction is easier: $(i)$ is the part of the Definition \ref{s.c. convergence}, while $(ii)$ again follows from equicontinuity of bounded family $f_n$ on a compact $K$.
\end{proof}

\subsection{A new take on Carath\'eodory convergence}
In this section, we give two more characterizations of Carath\'eodory convergence. The first one is purely geometric, and is strongly motivated by ideas of Computable Analysis applied to conformal mapping (cf. \cite{hertling}). Let us recall that the spherical metric on the Riemann sphere is given by
$$ds=\frac{2dz}{1+|z|^2};$$
it corresponds to the Euclidean metric on $\CC$ under the stereographic projection.

\begin{definition}\label{def:interior}Let $\Omega_n,\ n\in\NN$ be a sequence of domains in  $\hat \CC$, and $w_n\in\Omega_n$. Let also $\Omega$ be a domain containing $w$. We say that  the sequence $(\Omega_{n}, w_n)$,  and  $(\Omega, w)$ have \defin{arbitrarily good common interior approximations} if $w_n\to w$ and for every $ \eps > 0$, there exists $N \in \N$ and a closed connected set $K_{\eps} \subset\Omega \cap \bigcap_{n\geq N}\Omega_{n}$ containing $w$ such that
$$
\dist(x,\partial \Omega)<\eps \quad \text{ and } \quad \dist(x, \partial\Omega_n)<\eps
$$
holds  for all $x\in \partial K_{\eps}$ and all $n\geq N$, where the distance is taken in the spherical metric.
\end{definition}

\noindent
Let us  make the following observation:
\begin{lemma}\label{lem:independence}
Let $w'\in \Omega$ and let ${w'}_n\in\Omega_n$ with ${w'}_n\to w'$. Then the convergence in the sense of the Definition \ref{def:interior}
$$(\Omega_n,w_n)\to (\Omega,w)\text{ is equivalent to }(\Omega_n,{w'}_n)\to (\Omega,w').$$
\end{lemma}
\begin{proof}
Let $\gamma\subset \Omega$ be a simple closed curve joining $w$ to $w'$, and let $\eps_0$ be the distance from $\gamma$ to $\partial\Omega$. Then for all $\eps<\eps_0$ any $\eps$-interior approximation containing one of the points $w$, $w'$ would necessary contain the other one.
\end{proof}

\noindent
Our second characterization of Carath\'eodory convergence is quite concise:
\begin{definition}\label{def:measconv}
The sequence $(\Omega_n,w_n)$ converges to $(\Omega,w)$ \defin{in the sense of the harmonic measure} if
$$\omega_{\Omega_n}(w_n,\cdot)\to \omega_{\Omega}(w,\cdot)\text{ in the weak sense}.$$
\end{definition}

\begin{theorem}\label{thm:sc} For simply connected domains $\Omega_{n}\ni w_n$,  $n\in\NN$, and $\Omega\ni w$,  the following are equivalent:
\begin{itemize}
\item[(i)]$(\Omega_{n},w_n) \to (\Omega,w)$  in Carath\'eodory sense,
\item[(ii)] $(\Omega_{n}, w_n)$ and  $(\Omega, w)$ have arbitrarily good common interior approximations,
\item[(iii)] $(\Omega_n,w_n)\to(\Omega,w)$ in the sense of the harmonic measure.
\end{itemize}
\end{theorem}

\begin{proof}

(i) $\implies$ (ii).  Assume (i).  For all $\eps>0$ and $R>0$, let $K_{\eps}\subset \Omega$ be the closure of the connected component of the set
$$\left\{z\in\Omega\,:\,\dist(z, \partial \Omega)<\eps/4\right\}$$
containing $w$. Then $\dist(x, \partial \Omega)>\eps/2$ for all $x\in \partial K_{\eps}$.

Since $K_\eps$ is a compact subset of $\Omega$ (in $\hC$),  we know by (i) that there exists $N^*\in\NN$ such that $K_{\eps} \subset \bigcap_{n>N^{*}} \Omega_{n}$.  Let $x\in \partial K_{\eps}$ and suppose that there exists $n^*\geq N^*$ such   that $\dist(x, \partial  \Omega_{n^{*}})>\eps$.  This means that  the open set $B_{n^{*}}(x,\eps)\cup{\overset{\circ}{K}_{\eps}}$ contains $w$, is contained in $\Omega_{n^{*}}$, but is not contained in $\Omega$. By (i) again we know that there are at most finitely many such $n^{*}$, which shows (ii).\\

    (ii) $\implies$ (iii).  Let  $K_{\eps} \ni w$ be a common interior approximation. Note that such $K_{\eps}$ has non empty interior as long as $\eps<\dist(w_0,\partial \Omega)$, and hence contains $w_n$ for all large enough values of $n$. Let $T_w^{\eps}$ be the random variable defined as the first hitting time of $B_{t}$ on $\partial K_{\eps}$. Note that $$\dist(B_{T_w^{\eps}},\partial \Omega)<\eps.$$  Let $f$ be a 1-Lip function and let $M$ be a universal bound for the absolute value of $f$.  Since $\Omega$ is simply connected, it is, in particular, uniformly perfect.  Proposition~\ref{first hit} and the Strong Markov property of the Brownian motion imply that for any   $n$
    \begin{equation}\label{eq:hit}
    \PP(|B^y_{T_y}-y|\geq \eta)<\nu^n
    \end{equation}
    as long as $\dist(y,\partial\Omega)\le 2^{-n}\eta$. The rest of the proof of this implication follows from the following proposition.
    \begin{prop}\label{prop:convergence}
    Assume that for any $\delta>0$ there exists $\eps\in(0,\delta/10)$ such that
    $$\dist(y,\partial\Omega)<\eps\text{ implies  }\PP(|B^y_{T_y}-y|\geq \delta/10)<\delta/10M.$$
    Then
    $$
   W_1 [ \omega_{\overset{\circ}{K}_{\eps}}(w,\cdot),  \omega_{\Omega}(w,\cdot) ] < \delta/2.
   $$
    \end{prop}

\begin{proof}[Proof of the Proposition \ref{prop:convergence}]
    By using the strong Markov property of the Brownian motion again, we see that there exists $\eps < \delta/10$ for which the probability that $$|B_{T_w} - B_{T_w^{\eps}}| > \delta/10$$ is at most $\delta/10M$.

    We split the probabilities in the two complementary cases: one where $B_{T_w}$ stays $\delta/10$-close to $B_{T_w^\eps}$ and the complementary case. We have:
  \begin{align*}
    &| f(B_{T_w^\eps}) - \EE(f(B_{T_w} | B_{T_w^{\eps}}) |=  \\
		&|f(B_{T_w^\eps})-\EE[f(B_{T_w})|B_{T_w^\eps},\; |B_{T_w}-B_{T_w^\eps}|<\delta/10]\cdot \PP[|B_{T_w}-B_{T_w^\eps}|<\delta/10]+\\
		&|f(B_{T_w^\eps})-\EE[f(B_{T_w})|B_{T_w^\eps},\; |B_{T_w}-B_{T_w^\eps}|\geq \delta/10]\cdot \PP[|B_{T_w}-B_{T_w^\eps}|\geq \delta/10]<\\
		&\delta/10\cdot 1+M\cdot \delta/(10M)< \delta/2;
    \end{align*}
   which in turn implies
   $$
   |\EE_{B_{T_w^{\eps}}}(f(B_{T_w^{\eps}})) - \EE_{B_{T_w^{\eps}}}(f(B_{T_w}) | B_{T_w^{\eps}}) |  < \delta/2.
   $$
   Since $\EE_{B_{T_w^{\eps}}}(\EE(f(B_{T_w}) | B_{T_w^{\eps}})) = \EE(f(B_{T_w}))$, this implies
   $$
W_1 [ \omega_{\overset{\circ}{K}_{\eps}}(w,\cdot),  \omega_\Omega(w,\cdot) ] < \delta /2.
   $$
   \end{proof}
Since, by \eqref{eq:hit}, we can apply the proposition to both $\Omega$ and $\Omega_{n}$ with $n$ large enough, we also get
$$W_1 [ \omega_{\overset{\circ}{K}_{\eps}}(w,\cdot),  \omega_{\Omega}(w,\cdot) ] < \delta /2\text{ and }
W_1 [ \omega_{\overset{\circ}{K}_{\eps}}(w,\cdot),  \omega_{\Omega_n}(w_n,\cdot) ] < \delta /2.
   $$
We conclude that  $$W_1[\omega_{\Omega_n}(w_n,\cdot), \omega_\Omega(w,\cdot)] <\delta,$$ which proves (iii).

    (iii) $\implies$ (i). Assume  $$\omega_{n}\equiv\omega_{\Omega_n}(w_n,\cdot)\to \omega\equiv\omega_\Omega(w,\cdot)\text{ weakly}.$$  Let $0<d<\dist(w,\partial\Omega)$, then $\omega(D(w,d))=0$. Thus for any $\eps>0$ one can find $n$ large enough so that $\omega_n(D(w,d))<\eps$. Hence, by Beurling Projection Theorem, $D(w,d/2)\subset \Omega_n$, if $\eps$ is small enough.
		Thus the conformal maps $f_n=\phi_n^{-1}$ form a normal family. For any converging subsequence $(f_{n_k})\to g$ of $(f_n)$, the corresponding harmonic measures $\omega_{n_k}$ must converge weakly to the harmonic measure of the limiting domain $g(\DD)$, by the already established implication (i)$\implies$(iii). Thus, by our assumption, the harmonic measure on $g(\DD)$ is $\omega$. Since the harmonic measure is supported on the boundary of a domain, it determines the domain. Thus, $g=f$, and, by normality, $f_n$ converges to $f$ uniformly on compacts. Proposition~\ref{prop:fconv} completes the proof.
    \end{proof}

\noindent
Let us observe:
\begin{corollary}\label{cor:independence}
Let $w'\in \Omega$ and let ${w'}_n\in\Omega_n$ with ${w'}_n\to w'$. Then Carath\'eodory convergence
$$(\Omega_n,w_n)\to (\Omega,w)\text{ is equivalent to }(\Omega_n,{w'}_n)\to (\Omega,w').$$
\end{corollary}
\begin{proof}
If we use Definition~\ref{def:interior}, the corollary is just a restatement of Lemma \ref{lem:independence}.
\end{proof}


\section{Notions of convergence for arbitrary  domains in \texorpdfstring{$\hat\CC$}{{\bf C}}}
In this section, we consider sequences of arbitrary domains $\Omega_n\subset\hat\CC$ with marked points $w_n\in\Omega_n$. To simplify the discussion, we will always assume that these domains are regular.
This is the same as $\supp\omega_{\Omega_n}(w,\cdot)=\partial\Omega_n$ for any $w\in\Omega_n$.

Both of the definitions of convergence in the sense of harmonic measure (Definition \ref{def:measconv}) and of having arbitrarily good common interior approximations (Definition \ref{def:interior}) can be applied to this case without any changes. Yet, as we will see below, they are no longer equivalent.

First of all, the convergence in the sense of Definition \ref{def:interior} does not depend on the choice of the point $w\in\Omega$ and the corresponding points $w_n\in\Omega_n$, by Lemma \ref{lem:independence}.
In contrast,
 even for the sequences of regular planar domains, the notion of the convergence in  sense of harmonic measure depends on the choice of a point $w\in\Omega$:
\begin{example}
  \label{counterex1}
  There exists a sequence of planar domains $\Omega_n$ and a domain $\Omega$, two points $w_1,\ w_2\in\left(\cap_{n}\Omega_n\right)\cap\Omega$ such that the sequence $\omega_{\Omega_{n}}(w_1,\cdot)$ weakly converge to $\omega_{\Omega}(w_1,\cdot)$, but the sequence $\omega_{\Omega_{n}}(w_2,\cdot)$ does not weakly converge to $\omega_{\Omega}(w_2,\cdot)$.
\end{example}
\begin{proof}
  Namely, let us consider a $\Omega$ to be the unit disk $\DD$, $w_1=0$, $w_2=1/2$, and $$\Omega_n:=\DD\setminus \left\{1/2+r_ne^{i\theta},\ -\pi+r_n\leq\theta\leq\pi-r_n\right\},$$
where the sequence $r_n\in(0,1)$ is selected in such a way that $$\omega_{\Omega_n}(0,\left\{1/2+r_ne^{i\theta},\ -\pi+r_n\leq\theta\leq\pi-r_n\right\})<2^{-n}.$$ This can be done since when $r_n\to 0$, this harmonic measure tends to $0$. On the other hand,
$$\omega_{\Omega_n}(1/2,\left\{1/2+r_ne^{i\theta},\ -\pi+r_n\leq\theta\leq\pi-r_n\right\})\to 1\text{ as } n\to\infty.$$

Then  $\omega_{\Omega_n}(0,\cdot)$ converges to the normalized length on $S^1$, which is $\omega_{\DD}(0, \cdot)$, whereas $\omega_{\Omega_n}(1/2,S^1)\to 0$, so those measures do not converge to any measure supported on the boundary of $\DD$.
\end{proof}

We further have:
\begin{example}
  \label{counterex2}
There exist regular domains $\Omega_n$, $n\in\NN$, and two domains $\Omega\neq\Omega'$, $\partial\Omega\cap\partial\Omega'=\emptyset$ such that
$(\Omega_n,0)$ and $(\Omega',0)$ have arbitrarily good common interior approximations, but
$$\text{For any } w\in\Omega,\ \omega_{\Omega_n}(w,\cdot)\xrightarrow{w^*}\omega_{\Omega}(w,\cdot). $$
\end{example}
\begin{proof}
The domains $\Omega_n$ will be obtained from the unit disk by removing $2^n$ very small radial intervals around the circle $\{|z|=1/2\}$. More specifically,
$\Omega_n=\DD\setminus K_n$, where
$$K_n:=\bigcup_{k=1}^{2^n}\left[\left(1/2+2^{-n}-r_n\right)\exp\left(i\pi2k2^{-n}\right),\, \left(1/2+2^{-n}+r_n\right)\exp\left(i\pi2k2^{-n}\right)\right]$$
where $r_n$ are chosen to be so small that
$$\left\{z\,:\, \omega_{\Omega_n}(z, K_n)>2^{-n}\right\}\subset\left\{1/2+2^{-n-1}<|z|<1/2+2^{-n+1}\right\}=:A_n.$$

Let $\Omega=\DD$, $\Omega'=\frac12\DD$.

$(\Omega_n, 0)$ and $(\Omega', 0)$ have arbitrarily good common interior approximations (given by $\left(\frac12-2^{n}\right)\DD$). On the other hand, for any $w\in\DD$ there exists $N$ such that $w\not\in A_n$ for $n>N$, so $\omega_{\Omega_n}(w, K_n)\le2^{-n}$. This implies that $\omega_{\Omega_n}(w,\cdot)\xrightarrow{w^*}\omega_{\Omega}(w,\cdot).$
\end{proof}

Thus, the definitions \ref{def:interior} and \ref{def:measconv}  are indeed not generally comparable.
Nevertheless, we can formulate a condition for their equivalency.

\begin{theorem}\label{thm:planar}
Let $\Omega$ be a regular planar domain, and let $(\Omega_n)$ be a sequence of planar domains. Assume that $(\Omega_n)$ are \emph{uniformly regular}, i.e.  \begin{equation}\label{eq:uniform}\forall \delta>0\; \exists\eps>0\text{  such that } \dist(z,\partial\Omega_n)<\eps\implies\omega_{\Omega_n}(D(z, \delta))>1-\delta.
    \end{equation}
    Then the following are equivalent.
\begin{itemize}
\item[(i)] For any $w\in \Omega$ and \emph{any} sequence $w_n\in\Omega_n$ converging to $w$, the sequence $(\Omega_n, w_n)$ converges to $(\Omega, w)$ in the sense of harmonic measure.
\item[(ii)] For any $w\in \Omega$ and \emph{some} sequence $w_n\in\Omega_n$ converging to $w$, the sequence $(\Omega_n, w_n)$ converges to $(\Omega, w)$ in the sense of harmonic measure.
\item[(iii)] For some $w\in\Omega$ and some $w_n\in\Omega_n$ converging to $w$, the sequence $(\Omega_n, w_n)$ converges to $(\Omega, w)$ in the sense of Definition \ref{def:interior}.
\end{itemize}
\end{theorem}

Let us observe that by Proposition \ref{first hit} and \eqref{eq:hit}, a sequence of uniformly perfect domains (with the same constant of uniform perfectness) automatically satisfy this condition.
\begin{proof}
The implication $ (i)\implies (ii)$ is trivial.

 To prove the implication $(iii)\implies (i)$, let us first observe that the convergence in the sense of Definition \ref{def:interior} does not depend on the choice of the point $w\in\Omega$ and the sequence $w_n\in\Omega_n$, $w_n\to w$, by Lemma \ref{lem:independence}. $(i)$ is thus the direct consequence of Proposition \ref{prop:convergence} (and the fact that $\Omega$ itself is regular).

Let now, as in the proof of Theorem \ref{thm:sc},
 $\omega_{n}\equiv\omega_{\Omega_n}(w_n,\cdot)$,  $\omega\equiv\omega_\Omega(w,\cdot)$.

 To show $(ii)\implies (iii)$,   observe that if $r=\dist(w, \partial\Omega)$, then $$\lim_{n\to\infty}\omega_n(D(w,r/2))\leq \omega(D(w,r))=0.$$ Thus, by \eqref{eq:uniform}, one can find $r_0<r/4$ such that for large enough $n$, $$\dist(w,\partial\Omega_n)>2r_0.$$

 To finish the proof of the implication, let us show that
 \begin{prop}\label{prop:interior_compact}
   Every subsequence $(\Omega_{n_k})$ of $(\Omega_n, w_n)$ contains a subsequence converging in the sense of Definition \ref{def:interior} to some regular domain with the marked point $(\Omega', w)$.
 \end{prop}
 \begin{proof}[Proof of Proposition \ref{prop:interior_compact}]
 Let us select a finite collection of disks (in the spherical metric) $${\mathfrak D}_1:=(D(z_n, r_0/4))\text{ such that }\hC\subset\cup D(z_n, r_0/8)$$
 Let $K_1$ be the closure of the maximal connected union of elements of ${\mathfrak D}_1$ such that $r_0/4$-neighborhood of $K_1$, $D(z_n, r_0/2)$ is contained in the infinitely many domains from $(\Omega_{n_k})$ and such that $w\in K_1$.  For these domains, $K_1$ form an $r_0$-interior approximation. Also, since $\dist(w,\partial\Omega_{n_k})>2 r_0$ for large $k$, $K_1$ is not empty.

 Repeating this for some finite collections
  $${\mathfrak D}_l:=(D(z_n, 2^{-l-1}r_0))\text{ such that }\hC\subset\cup D(z_n, 2^{-l-2}r_0)$$
 and using the diagonal process, we
 obtain an increasing sequence $K_l$ and a sequence $n_{k_j}$,  such that $K_l$ serve as an interior approximation for all but finitely many $\Omega_{n_{k_j}}$. Let us take $\Omega'=\cup K_l$. By the condition \eqref{eq:uniform}, $\Omega'$ is regular and the limit in the sense of Definition \ref{def:interior}.
  \end{proof}
By Proposition~\ref{prop:interior_compact} and the already established implication $(iii)\implies (i)$, the corresponding subsequence $\omega_{n_k}$ converges to  $\omega_{\Omega'}(w,\cdot)$. So $\omega=\omega_{\Omega'}(w,\cdot)$, and thus $\Omega=\Omega'$. Hence, $(\Omega_n, w_n)$ converges to $(\Omega,w)$ in the sense of Definition \ref{def:interior}.
\end{proof}
\section{Notions of convergence for domains in \texorpdfstring{$\RR^d$, $d\geq3$}{{\bf R}d, d>=3}}

In this section, we consider sequences of domains $\Omega_n$ in $\RR^d$, $d\geq3$ with marked points $w_n\in\Omega_n$. As in the previous section, we will always assume that $\partial\Omega_n$ is the closure of its set of regular points.

While the definitions of convergence in the sense of harmonic measure (Definition \ref{def:measconv}) works in this case without any changes, we need to modify Definition \ref{def:interior} to take care of the behavior at infinity.

\begin{definition}\label{def:interior_hiD} Let $\Omega_n,\ n\in\NN$ be a sequence of domains in  $\RR^d$, and $w_n\in\Omega_n$. Let also $\Omega$ be a domain containing $w$. We say that  the sequence $(\Omega_{n}, w_n)$,  and  $(\Omega, w)$ have \defin{arbitrarily good common interior approximations} if $w_n\to w$ and for every $ \eps > 0$ and $R>0$, there exists $N \in \N$ and a closed connected set $K_{\eps, R} \subset\Omega \cap \bigcap_{n\geq N}\Omega_{n}\cap D(0, R)$ containing $w$ such that
$$
\dist(x,\partial \left(\Omega\cap D(0, R)\right))<\eps \quad \text{ and } \quad \dist(x, \partial\left(\Omega_n\cap D(0, R)\right))<\eps
$$
holds  for all $x\in \partial K_{\eps, R}$ and all $n\geq N$, where the distance is taken in the standard Euclidean metric.
\end{definition}

\noindent
Note that Lemma \ref{lem:independence} still holds for this modified definition of convergence, with the same proof.
\\
We have an analogue of Theorem \ref{thm:planar}:

\begin{theorem}\label{thm:general}
Let $\Omega$ be a regular domain, and let $(\Omega_n)$ be a sequence of domains in $\RR^d$. Assume that $(\Omega_n)$ are \emph{uniformly regular} (in the sense of \eqref{eq:uniform}) and for some sequence $w_n\in\Omega_n$, $w_n\to w\in\Omega$, we have
\begin{equation}\label{eq:infty}
\forall\eps>0\exists R>0, N\in\NN\text{ such that } \omega_{\Omega_n\cap D(0,2R)}(w_n, D(0, R))>1-\eps
\end{equation}
    Then the following are equivalent.
\begin{itemize}
\item[(i)] For any $w\in \Omega$ and \emph{any} sequence $w_n\in\Omega_n$ converging to $w$, the sequence $(\Omega_n, w_n)$ converges to $(\Omega, w)$ in the sense of harmonic measure.
\item[(ii)] For any $w\in \Omega$ and \emph{some} sequence $w_n\in\Omega_n$ converging to $w$, the sequence $(\Omega_n, w_n)$ converges to $(\Omega, w)$ in the sense of harmonic measure.
\item[(iii)] For some $w\in\Omega$ and some $w_n\in\Omega_n$ converging to $w$, the sequence $(\Omega_n, w_n)$ converges to $(\Omega, w)$ in the sense of Definition \ref{def:interior}.
\end{itemize}
\end{theorem}
\begin{proof}
As before, the implication $(i)\implies(ii)$ is trivial.

Formally, we can no longer use Proposition \ref{prop:convergence} to establish $(iii)\implies (i)$, in particular, because we can no longer use Wasserstein-Kantorovich distance.

Instead, we fix a function $g$ with compact support and $\delta>0$. Let $M:=\|g\|_\infty$.

Observe that for uniformly regular sequences of domains, the condition \eqref{eq:infty} does not depend on the choice of $w_n$ or $w$.
Thus by \eqref{eq:infty}, we can find $R$ such that
$$\omega_{\Omega_n\cap D(0,2R)}(w_n, D(0, R))>1-\frac{\delta}{10M}\text{ and }\omega_{\Omega\cap D(0,2R)}(w, D(0, R))>1-\frac{\delta}{10M}.$$
We can also select $\eta>0$ so that
$$\|x_1-x_2\|<\eta\implies |g(x_1)-g(x_2)|<\frac\delta{10}.$$
Finally, fix an $\eps>0$ so that
$$\dist(z,\partial\Omega_n)<\eps\implies\omega_{\Omega_n}(D(z, \delta))>1-\frac\delta{10}$$
and
$$\dist(z,\partial\Omega)<\eps\implies\omega_{\Omega}(D(z, \eta))>1-\frac\delta{10}.$$

Let  $K_{\eps} \ni w$ be a common interior approximation corresponding to the just selected $\eps$ and $R$.Let us use the notations from this proof of Proposition \ref{prop:convergence}. The same stopping time reasoning as in that proof shows that, provided $n$ is large enough
\begin{multline}
\left|\int g\,d\omega_{\Omega}(w, \cdot)-\int g\,d\omega_{\Omega_n}(w_n, \cdot)\right|\leq\\ \EE\left[|g(B_{T_w}) - g(B_{T_w^{\eps}})|\right]+\EE\left[|g(B_{T^n_w}) - g(B_{T_w^{\eps}})|\right]<\delta.
\end{multline}
Since it holds for an arbitrary compactly supported continuous $g$ and $\delta>0$, $(i)$ follows.

Finally, the proof of the implication $(ii)\implies(iii)$ is the same as in the proof of Theorem \ref{thm:planar}.
\end{proof}

\end{document}